\documentclass[oneside,12pt]{amsart}
\usepackage{amssymb, mathrsfs, amscd}
\usepackage[all]{xy}
\usepackage{enumerate}

\newtheorem{theorem}{Theorem}[section]
\newtheorem{definition}[theorem]{Definition}
\newtheorem{lemma}[theorem]{Lemma}

\newtheorem{proposition}[theorem]{Proposition}

\newtheorem{remark}[theorem]{Remark}
\newtheorem{example}[theorem]{Example}

\author{Dieter Degrijse}
\address{Department of Mathematics, Catholic University of Leuven, Kortrijk, Belgium}%
\email{Dieter.Degrijse@kuleuven-kortrijk.be}%

\author{Nansen Petrosyan}
\address{Department of Mathematics, Catholic University of Leuven, Kortrijk, Belgium}%
\email{Nansen.Petrosyan@kuleuven-kortrijk.be}%
\title{On cohomology of split Lie algebra extensions}

\thanks{The second author was supported by FWO-Flanders Research Fellowship.}
\date{\today}

\begin{document}
\begin{abstract}
\noindent We introduce the notion of compatible actions in the context of split extensions of Lie algebras over a field $k$. Using compatible actions, we construct new resolutions to compute the cohomology of semi-direct products of Lie algebras and give an alternative way to construct the Hochschild-Serre spectral sequence associated to a split extension. Finally, we describe several instances in which this spectral sequence collapses at the second page and obtain a sharper bound for its length in the finite dimensional case.
\end{abstract}
\maketitle
\section{Introduction}
In \cite{Evens}, L. Evens constructed a resolution to compute the cohomology of the semi-direct product $H \rtimes G$ of two groups.  This resolution arose by considering a special action of $G$ on a free resolution for $H$. The construction was later made explicit by T. Brady in \cite{Brady} where he named it a {\it compatible action}.
This approach has proven to be very useful for computing the cohomology of certain semi-direct product groups such as crystallographic groups  (see for example \cite{AdemPan} and \cite{AdemPanGePetr}). \\
\indent In this paper, we define the analogue of compatible group actions in the context of Lie algebras over a field $k$. More concretely, we consider
a split extension of Lie algebras
$ 0 \to \mathfrak{n} \to \mathfrak{g} \to \mathfrak{h} \to 0 $
over a field $k$, a free resolution  $P \rightarrow k$ for $\mathfrak{h}$ and a free resolution $F \to k$ for $\mathfrak{n}$. Then we define the notion of compatible action in such a way that, if $\mathfrak{h}$ acts compatibly on $F$, we can define a $\mathfrak{g}$-module structure on $P \otimes_k F$ that turns this complex into a free resolution for $\mathfrak{g}$. Using this fact, we obtain an alternative way to construct the Hochschild-Serre spectral sequence of a split Lie algebra extension, from which we derive the following.

\begin{theorem} \label{th: introduction3}
Suppose $0 \to \mathfrak{n} \to \mathfrak{g} \to \mathfrak{h} \to 0 $
is a split extension of Lie algebras. Let M be a $\mathfrak{g}$-module and denote by $(E_r,d_r)$ the associated Hochschild-Serre spectral sequence.
If $\mathfrak{h}$ acts compatibly on a free  $U(\mathfrak{n})$-resolution $F$ such that the differential
\[ d^{q-1}: \mathrm{Hom}_{\mathfrak{n}}(F_{q-1},M) \rightarrow \mathrm{Hom}_{\mathfrak{n}}(F_q,M) \]
is zero, then $d^{p,q}_r$ and $d^{p,q+r-2}_r$ are zero for all $p$ and all $r\geq 2$.
\end{theorem}
The accessibility of this construction, of course, depends on the fact whether a particular resolution for $\mathfrak{n}$ admits a compatible action of $\mathfrak{h}$. As it turns out, $\mathfrak{h}$ always acts compatibly on the Chevalley-Eilenberg complex of $\mathfrak{n}$.  This allows us to form a
practical cochain complex for computing the cohomology of $\mathfrak{g}$. By explicitly constructing compatible actions, we obtain a collapse of the Hochschild-Serre spectral sequence in the following cases.
\begin{theorem} Consider the split extension $0 \rightarrow \mathfrak{n} \rightarrow \mathfrak{g} \rightarrow \mathfrak{h} \rightarrow 0$
determined by $\varphi: \mathfrak{h} \rightarrow \mathrm{Der}(\mathfrak{n})$. Let $M$ be a $\mathfrak{g}$-module with a trivial ${\mathfrak{n}}$-action. Then the Hochschild-Serre spectral sequence associated to this extension with coefficients in $M$ collapses at $E_2$ in the following cases
\begin{itemize}
\item[(a)] $\mathfrak{n}=\mathfrak{n}_1 \oplus \mathfrak{n}_2$, where $\mathfrak{n}_1$ is either abelian or free and $\mathfrak{n}_2$ is either abelian or free;

\smallskip

\item[(b)] $\mathfrak{n}=\mathfrak{n}_1 * \mathfrak{n}_2* \ldots * \mathfrak{n}_k$, where each of the $\mathfrak{n}_i$ is either abelian or free and $\varphi(\alpha)(\mathfrak{n}_i)\subseteq \mathfrak{n}_i$ $\forall \alpha \in \mathfrak{h}$ and $i=1,\ldots, k$.
\end{itemize}
\end{theorem}
\indent In \cite{Barnes}, D. Barnes showed that the length  $l$ of the Hochschild-Serre spectral sequence associated to a split extension of finite dimensional Lie algebras with kernel $\mathfrak{n}$  satisfies  $l \leq \max{\{2,\dim_k(\mathfrak{\mathfrak{n}})\}}$ when $\mathfrak{n}$ is nilpotent and acts trivially on the coefficient space. As another corollary of Theorem \ref{th: introduction3}, we prove the following generalization of this result.

\begin{theorem} Suppose  $0 \rightarrow \mathfrak{n} \rightarrow \mathfrak{g} \rightarrow \mathfrak{h} \rightarrow 0 $
is a split extension of Lie algebras such that $\dim_k(\mathfrak{\mathfrak{n}})=m < \infty$. Denote by $(E_r,d_r)$ the associated Hochschild-Serre spectral sequence with coefficients in a $\mathfrak{g}$-module $M$.
If $\mathfrak{n}$ acts trivially on $M$, then
\begin{itemize}
\item[(a)] $d_r^{p, m}=0$ for all $p$ and all $r\geq 2$;

\smallskip

\item[(b)] $l \leq \max{\{2, m\}}$;

\smallskip

\item[(c)] $\mathrm{H}^p(\mathfrak{h}, \mathrm{H}^{m}(\mathfrak{n},M))\oplus\mathrm{H}^{p+m}(\mathfrak{h}, M)\subseteq \mathrm{H}^{p+m}(\mathfrak{g},M)$ for all $p$.
\end{itemize}
\end{theorem}

\section{Definitions, Notations and preliminary results}
\indent Suppose $R$ is a ring with unit, and let $(A,d^h,d^v)$ be a double complex of $R$-modules. We define the total complex $\mathscr{A}$ to be the chain complex with
$ \mathscr{A}_n = \bigoplus_{k+l=n}A_{k,l}$
and differential $d$ defined by $d^h+d^v$.

Now, let $(P,d)$ be a chain complex of right $R$-modules and let $(Q,d')$ be a chain complex of left $R$-modules. Then, we define the double complex $(B,d^h,d^v)$ as
$B_{p,q}=P_p\otimes_R Q_q$
\begin{align*}
d^h_{p,q}: B_{p,q} \rightarrow B_{p-1,q}, & \hspace{3mm} x\otimes y \mapsto d_p(x)\otimes y \\
d^v_{p,q}: B_{p,q} \rightarrow B_{p,q-1}, & \hspace{3mm} x\otimes y \mapsto (-1)^px\otimes d'_q(y).
                                \end{align*}
We define the tensor product of $P$ and $Q$ to be $\mathscr{B}$. In the future we will denote $B$ and $\mathscr{B}$ both by $P\otimes_R Q$; the meaning will be apparent from the context.

When $(P,d)$ is a chain complex of left $R$-modules and $(Q,d')$ is  a cochain complex of left $R$-modules, we define the double complex $(C,d_h,d_v)$ as
$C^{p,q}=\mathrm{Hom}_R(P_p,Q^q)$
\begin{align*}
 d_h^{p,q}: C^{p,q} \rightarrow C^{p+1,q}, & \hspace{3mm} f \mapsto f\circ d_{p+1}  \\
 d_v^{p,q}: C^{p,q} \rightarrow C^{p,q+1},  & \hspace{3mm} f \mapsto (-1)^pd'^q\circ f.
                                \end{align*}
We denote the total Hom cochain complex of $P$ and $Q$ by $\mathscr{C}$. Like before, we will abuse notation and denote both $C$ and $\mathscr{C}$ by $\mathrm{Hom}_R(P,Q)$. \\
\indent All Lie algebras we consider are over a fixed field $k$. Let $\mathfrak{g}$ be a Lie algebra. If $M$ and $N$ are $\mathfrak{g}$-modules then $M\otimes_k N$ and $\mathrm{Hom}_k(M,N)$ naturally become $\mathfrak{g}$-modules in the following way
\begin{align*}
\alpha(m \otimes n)&=\alpha m\otimes n + m \otimes \alpha n,  \hspace{3mm}  \alpha \in \mathfrak{g}, m \in M, n \in N;  \\
 (\alpha f)(m)&= \alpha f(m)-f(\alpha m),   \hspace{3mm}  \alpha \in \mathfrak{g}, m \in M, f \in \mathrm{Hom}_k(M,N).
                                \end{align*}
Some useful properties of these $\mathfrak{g}$-module structures are summarized in the following lemma.
\begin{lemma} \label{lem: gmod lemma}
There is a natural isomorphism
$ \mathrm{Hom}_k(M,N)^{\mathfrak{g}}\cong \mathrm{Hom}_{\mathfrak{g}}(M,N)$.
Also, the functor $\mathrm{Hom}_k(N,-): \mathfrak{g}\mbox{-mod}  \rightarrow \mathfrak{g} \mbox{-mod}$
is right adjoint to the functor
 $-\otimes_k N: \mathfrak{g}\mbox{-mod}   \rightarrow \mathfrak{g}\mbox{-mod}$,
which implies that there exists a natural isomorphism
\[ \mathrm{Hom}_{\mathfrak{g}}(M \otimes_k N,K) \cong \mathrm{Hom}_{\mathfrak{g}}(M,\mathrm{Hom}_k(N,K))\]
for all $\mathfrak{g}$-modules $M,N$ and $K$.
\end{lemma}

Denote by $U(\mathfrak{g})$ the universal enveloping algebra of $\mathfrak{g}$. Note that the category of $\mathfrak{g}$-modules is naturally isomorphic to the category of $U(\mathfrak{g})$-modules, so we will identify them without mentioning. The cohomology of $\mathfrak{g}$ with coefficients in the $\mathfrak{g}$-module $M$ is defined as
$\mathrm{H}^{\ast}(\mathfrak{g},M)=\mathrm{Ext}^{\ast}_{U(\mathfrak{g})}(k,M)$. Hence, $\mathrm{H}^{\ast}(\mathfrak{g},M)$ can be computed by taking the cohomology of $\mathrm{Hom}_{\mathfrak{g}}(F,M)$, where $F$ is any free $U(\mathfrak{g})$-resolution of $k$.
For details on homological algebra and the cohomology of Lie algebras, we refer the reader to \cite{Knapp} and \cite{Weibel}.

\begin{lemma}\label{lem: gmod lemma2} Let $0 \rightarrow \mathfrak{n} \rightarrow \mathfrak{g} \xrightarrow{\pi} \mathfrak{h} \rightarrow 0$ be short exact sequence of Lie algebras. If $K,N$ are $\mathfrak{g}$-modules such that $\mathfrak{n}$ acts trivially on $K$, then there is a natural isomorphism
\[ \mathrm{Hom}_{\mathfrak{g}}(K,N) \cong \mathrm{Hom}_{\mathfrak{h}}(K,N^{\mathfrak{n}}). \]
In particular, we have a natural isomorphism of functors $-^{\mathfrak{g}} \cong -^{\mathfrak{h}}\circ -^{\mathfrak{n}}$,
where we consider $-^{\mathfrak{n}}$ as a functor from $\mathfrak{g}$-mod to $\mathfrak{h}$-mod.
\end{lemma}

\section{Compatible Actions}
We are especially interested in  split short exact sequences of Lie algebras over a field $k$

\begin{equation}\label{eq: splitextension}\xymatrix {&0 \ar[r] &\mathfrak{n}\ar[r] &\mathfrak{g} \ar[r]^{\pi} &\mathfrak{h}
\ar@/_1.5pc/[l]\ar[r] &0}.
\end{equation} There is a Lie algebra homomorphism
$\varphi: \mathfrak{h} \rightarrow \mathrm{Der}(\mathfrak{n})$, where $\mathrm{Der}(\mathfrak{n})$ is the derivation algebra of $\mathfrak{n}$.
Using $\varphi$, we can write $\mathfrak{g}$ as a semi-direct product
$ \mathfrak{g}= \mathfrak{n} \rtimes_{\varphi} \mathfrak{h}$.
Viewed this way,  multiplication in $\mathfrak{g}$ is given by
\[ [(s,\alpha),(t,\beta)]=([s,t]+\varphi(\alpha)(t)-\varphi(\beta)(s),[\alpha,\beta]), \ \ \ \forall \alpha,\beta \in \mathfrak{h}, \ s, t \in \mathfrak{n}. \]
In what follows, we will drop $\varphi$ from our notation and write $\varphi(\alpha)(t)$ as $\alpha(t)$ for all $\alpha \in \mathfrak{h}$ and $t \in \mathfrak{n}$. Given a  $\mathfrak{g}$-module $M$,
we will construct a new resolution to compute $\mathrm{H}^{\ast}(\mathfrak{g},M)$. Our result will depend on the existence of what is called a \emph{compatible action}.
\begin{definition} \label{def: comp act} \rm Suppose $\varepsilon :F \rightarrow k$ is a free resolution of $k$ over $U(\mathfrak{n})$. Let $\mathfrak{C}(F)$ be the set of chain maps from $F$ to itself that extend the zero map on $k$. It is an associative $k$-algebra under composition and hence it can be given the standard Lie algebra structure. We say $\mathfrak{h}$ \emph{acts compatibly on} $F$, if there exists a Lie algebra homomorphism $\Theta: \mathfrak{h} \rightarrow \mathfrak{C}(F): \alpha \mapsto \underline{\alpha}$, such that
\begin{equation} \label{eq: comp cond}
\alpha(s)f=\underline{\alpha}(sf) - s\underline{\alpha}(f)
\end{equation}
\noindent for all $\alpha \in \mathfrak{h}$, $s \in \mathfrak{n}$ and $f \in F_{\ast}$.
\end{definition}
Given an $\mathfrak{h}$-module $M$, we can use the projection map $\pi: \mathfrak{g} \rightarrow \mathfrak{h}$ to turn $M$ into a $\mathfrak{g}$-module. Moreover, a $U(\mathfrak{h})$-resolution of $k$ inflates to a $U(\mathfrak{g})$-resolution of $k$. However, since the projection of $\mathfrak{g}$ onto $\mathfrak{n}$
is not a Lie algebra homomorphism, there is no obvious way of extending a $U(\mathfrak{n})$-resolution to a $U(\mathfrak{g})$-resolution. This is where compatible actions come into play.
\begin{lemma} \label{prop: gmod}
Suppose there is a compatible action of $\mathfrak{h}$ on a $U(\mathfrak{n})$-resolution $\varepsilon :F \rightarrow k$. Let $(s,\alpha) \in \mathfrak{g}$ for $s \in \mathfrak{n}$, $\alpha \in \mathfrak{h}$, and $f \in F_{\ast}$, then
\begin{equation}\label{eq: induced gmod structure} (s,\alpha)f = sf+ \underline{\alpha}(f) \end{equation}
turns $F\rightarrow k$ into a resolution of $U(\mathfrak{g})$-modules.
\end{lemma}
\begin{proof}For each $n$, denote by $F_n$ the $n^{\mathrm{th}}$-module of $F$. By definition of compatible action, the action in (\ref{eq: induced gmod structure}) turns $F_n$ into a $\mathfrak{g}$-module.

To see that the differentials of $F$ are $\mathfrak{g}$-module homomorphisms, we use the fact that $\underline{\alpha}$ is a chain map for each $\alpha \in \mathfrak{h}$. Let $f \in F_n$ and $(s,\alpha) \in \mathfrak{g}$, then
\[ d((s,\alpha)f)=d(sf+\underline{\alpha}(f))= sd(f)+ \underline{\alpha}(d(f)) = (s,\alpha)d(f).\]
Finally, the augmentation $\varepsilon: F_0 \rightarrow k$ becomes a $\mathfrak{g}$-module map (give $k$ trivial $\mathfrak{g}$-module structure) because $\underline{\alpha}$ extends the zero map on $k$ for each $\alpha \in \mathfrak{h}$. Let $f \in F_0$ and $(s,\alpha) \in \mathfrak{g}$. Then, we have
\[
\varepsilon((s,\alpha)f) =  s\varepsilon(f) + \varepsilon(\underline{\alpha}(f))= 0 = (s,\alpha)\varepsilon(f). \]
\end{proof}
\indent Next, we consider a free $U(\mathfrak{n})$-resolution $\varepsilon_F :F \rightarrow k$ and assume that it admits a compatible action of $\mathfrak{h}$. Using Lemma \ref{prop: gmod}, we inflate $\varepsilon_F: F \rightarrow k$ into a (not necessarily free) $U(\mathfrak{g})$-resolution of $k$. Also, we consider a free $U(\mathfrak{h})$-resolution $\varepsilon_P:P \rightarrow k$ of $k$ and turn it into a $U(\mathfrak{g})$-resolution of $k$, using the projection map $\pi$.
The complex $P\otimes_k F$ now turns out to be a free resolution of $U(\mathfrak{g})$-modules. To summarize, we have
\begin{lemma} \label{prop: freeres}The complex $\varepsilon_P \otimes \varepsilon_F: P\otimes_k F \rightarrow k$
 is a free $U(\mathfrak{g})$-resolution, with the action of $U(\mathfrak{g})$ on  $P\otimes_k F$  induced by
 $$(s,\alpha)(p\otimes f) := \alpha p\otimes f + p\otimes (sf+ \underline{\alpha}(f))$$
 for each $(s,\alpha)\in \mathfrak{g}$, $p\in P_*$, and $f\in F_*$.
\end{lemma}
\begin{proof}
From the K\"{u}nneth formula for tensor products, it follows that $\varepsilon_P \otimes \varepsilon_F: P\otimes_k F \rightarrow k$ is a $U(\mathfrak{g})$-resolution of $k$.

The $n^{\mathrm{th}}$-module of $P\otimes_k F$ is given by
$ \bigoplus_{p+q=n}P_q \otimes_k F_q$, and we need to show that this is a free $U(\mathfrak{g})$-module. Because $P$ consists of free $U(\mathfrak{h})$-modules,  it suffices to show that
$ U(\mathfrak{h})\otimes_k F_q$
is a free $U(\mathfrak{g})$-module for every $q$. Furthermore, it follows from tensor identities that the $\mathfrak{g}$-modules
$U(\mathfrak{h})\otimes_k F_q $ and $U(\mathfrak{g}) \otimes_{U(\mathfrak{n})} F_q$ are isomorphic,
where the $\mathfrak{g}$-module structure on  $U(\mathfrak{g}) \otimes_{U(\mathfrak{n})} F$ is given by multiplication on the left in $U(\mathfrak{g})$. Hence, we see that
$$U(\mathfrak{h})\otimes_k F_q  \cong  U(\mathfrak{g}) \otimes_{U(\mathfrak{n})} F_q
 \cong  U(\mathfrak{g}) \otimes_{U(\mathfrak{n})} \Big(\oplus_{i\in I}U(\mathfrak{n})\Big) \cong  \oplus_{i \in I} U(\mathfrak{g}).$$\end{proof}
\section{Constructing compatible actions}
We will first show that compatible actions always exist for the Chevalley-Eilenberg complex $V(\mathfrak{n})$ of $\mathfrak{n}$:
\[ \ldots\rightarrow U(\mathfrak{n})\otimes_k \Lambda^{p}(\mathfrak{n}) \xrightarrow{d_p}\ldots \xrightarrow{d_2}  U(\mathfrak{n})\otimes_k \Lambda^{1}(\mathfrak{n})\xrightarrow{d_1}  U(\mathfrak{n})\otimes\Lambda^{0}(\mathfrak{n})\xrightarrow{\varepsilon} k \rightarrow 0.\]
where $\Lambda^{p}(\mathfrak{g})$ denotes the $p$-th exterior product of $\mathfrak{g}$, $\varepsilon$ is the usual augmentation map, and $d_1:U(\mathfrak{n})\otimes_k \mathfrak{n}\to  U(\mathfrak{n})$ is the product map $d_1(u\otimes x)=ux$. For $p\geq 2$, and $u\otimes x_1\wedge \ldots \wedge x_p \in V_p(\mathfrak{n})$, ($u\in U(\mathfrak{n})$, $x_i\in \mathfrak{n}$) the boundary map is given by
\begin{align*}\label{eq: diff}
 d_p(u\otimes x_1\wedge \ldots \wedge x_p) =&\sum_{i=1}^p (-1)^{i+1}ux_i  x_1 \wedge \ldots \wedge \hat{x_i} \wedge \ldots \wedge x_{p} +\\
& \sum_{i < j}(-1)^{i+j} u\otimes [x_i,x_j] \wedge x_1 \wedge \ldots \wedge \hat{x_i} \wedge \ldots \wedge \hat{x_j} \wedge \ldots \wedge x_{p}.\\
\end{align*}

\begin{proposition} \label{prop: comp act on chev}
Given the split extension (\ref{eq: splitextension}),  the maps
\begin{eqnarray}
\underline{\alpha}:  U(\mathfrak{n})\otimes_k \Lambda^{p}(\mathfrak{n}) &\rightarrow &U(\mathfrak{n})\otimes_k \Lambda^{p}(\mathfrak{n}):\nonumber \\
1 \otimes  x_{1}\wedge \ldots \wedge  x_{p} & \mapsto &\sum_{j=1}^p 1 \otimes x_{1}\wedge \ldots \wedge\alpha(x_{j})\wedge \ldots \wedge  x_{p} \nonumber, \\
 y_1 \ldots y_m \otimes x_{1}\wedge \ldots \wedge  x_{p} & \mapsto & \sum_{j=1}^m y_1\ldots \alpha(y_j)\ldots y_m \otimes x_{1}\wedge \ldots \wedge  x_{p}\nonumber,\\
& & + \sum_{j=1}^p y_1 \ldots y_m \otimes x_{1}\wedge \ldots \wedge\alpha(x_{j})\wedge \ldots \wedge  x_{p} \nonumber
\end{eqnarray}
for all $\alpha \in \mathfrak{h}$, define a compatible action of $\mathfrak{h}$ on the Chevalley-Eilenberg complex of $\mathfrak{n}$. (If $p=0$, then the second big sum disappears.)
\end{proposition}
\begin{proof} Let us first show that for each $\alpha \in \mathfrak{h}$, $\underline{\alpha}$ is an augmentation preserving chain map. By simple computations, this reduces to showing that
$d\circ \underline{\alpha}(1\otimes x_{1}\wedge \ldots \wedge  x_{p})= \underline{\alpha} \circ d (1\otimes x_{1}\wedge \ldots \wedge  x_{p}),$
for all $p$.
First, we compute the left hand side (L). \begin{eqnarray*}
(\mathrm{L}) & = & \sum_{j=1}^p d(1 \otimes x_{1}\wedge \ldots \alpha(x_{j}) \ldots \wedge  x_{p})\\
& = & \sum_{j=1}^p (-1)^{j+1}\alpha(x_{j}) \otimes x_{1}\wedge \ldots \wedge \hat{x}_{j} \wedge \ldots \wedge  x_{p}  \\
& & +\sum_{\substack{l,j=1 \\ l \neq j}}^p (-1)^{l+1}x_{l}\otimes x_{1}\wedge \ldots \wedge \hat{x}_{l} \wedge \ldots \wedge \alpha(x_{j}) \wedge  \ldots \wedge  x_{p}\\
& & + \sum_{l>j}^p(-1)^{l+j}\otimes [\alpha(x_{j}),x_{l}]\wedge x_{1} \wedge \ldots \wedge \hat{x}_{j} \wedge \ldots \wedge \hat{x}_{l}\wedge \ldots \wedge x_{p} \\
& &  +  \sum_{j>l}^p(-1)^{l+j}\otimes [x_{l},\alpha(x_{j})]\wedge x_{1} \wedge \ldots \wedge \hat{x}_{l} \wedge \ldots \wedge \hat{x}_{j}\wedge \ldots \wedge x_{p} \\
& &  + \sum_{j=1}^p \sum_{\substack{l>k \\ l \neq j \neq k}}^p(-1)^{l+k}\otimes [x_{k},x_{l}]\wedge x_{1} \wedge \ldots \wedge \hat{x}_{k} \wedge \ldots \wedge \hat{x}_{l}\wedge \ldots \wedge \alpha(x_{j}) \wedge \ldots \wedge x_{p}.
\end{eqnarray*} Since $\alpha$ acts as a derivation, we have $\alpha([x_{l},x_{j}])=[\alpha(x_{l}),x_{j}]+[x_{l},\alpha(x_{j})]$. So, continuing with the equality, we find \begin{eqnarray*}
(\mathrm{L}) & = & \sum_{j=1}^p (-1)^{j+1}\alpha(x_{j}) \otimes x_{1}\wedge \ldots \wedge \hat{x}_{j} \wedge \ldots \wedge  x_{p}  \\
& & + \sum_{\substack{l,j=1 \\ l \neq j}}^p (-1)^{j+1}x_{j}\otimes x_{1}\wedge \ldots \wedge \hat{x}_{j} \wedge \ldots \wedge \alpha(x_{l}) \wedge  \ldots \wedge  x_{p}\\
&& + \sum_{j>l}^p(-1)^{l+j}1 \otimes \alpha([x_{l},x_{j}])\wedge x_{1} \wedge \ldots \wedge \hat{x}_{l} \wedge \ldots \wedge \hat{x}_{j}\wedge \ldots \wedge x_{p} \\
& & + \sum_{j=1}^p \sum_{\substack{l>k \\ l \neq j \neq k}}^p(-1)^{l+k}\otimes [x_{k},x_{l}]\wedge x_{1} \wedge \ldots \wedge \hat{x}_{k} \wedge \ldots \wedge \hat{x}_{l}\wedge \ldots \wedge \alpha(x_{j}) \wedge \ldots \wedge x_{p}. \\
\end{eqnarray*}
\noindent Meanwhile, the right hand side ($\mathrm{R}$) is
\begin{eqnarray*}
(\mathrm{R}) & = & \sum_{j=1}^p(-1)^{j+1}\underline{\alpha}(x_{j}\otimes  x_{1}\wedge \ldots \wedge \hat{x}_{j} \wedge \ldots \wedge  x_{p} ) \\
& & + \sum_{j>l}^p(-1)^{l+j}\underline{\alpha}(1\otimes [x_{l},x_{j}]\wedge x_{1} \wedge \ldots \wedge \hat{x}_{l} \wedge \ldots \wedge \hat{x}_{j}\wedge \ldots \wedge x_{p}) \\
& = & \sum_{j=1}^p (-1)^{j+1}\alpha(x_{j}) \otimes x_{1}\wedge \ldots \wedge \hat{x}_{j} \wedge \ldots \wedge  x_{p}  \\
& & + \sum_{\substack{l,j=1 \\ l \neq j}}^p (-1)^{j+1}x_{j}\otimes x_{1}\wedge \ldots \wedge \hat{x}_{j} \wedge \ldots \wedge \alpha(x_{l}) \wedge  \ldots \wedge  x_{p} \\
& & + \sum_{j<l}^p(-1)^{l+j}\underline{\alpha}(1\otimes [x_{l},x_{j}]\wedge x_{1} \wedge \ldots \wedge \hat{x}_{l} \wedge \ldots \wedge \hat{x}_{j}\wedge \ldots \wedge x_{p}). \\
\end{eqnarray*}
Now, using the definition of $\underline{\alpha}$, we see that this is the same expression as before.

Next, straightforward computations confirm that the  map $\Theta: \mathfrak{h} \rightarrow \mathfrak{C}(V(\mathfrak{n}))$ is a Lie algebra homomorphism.

It is left to check condition (\ref{eq: comp cond}).  Suppose $ y_1y_2\ldots y_m \in U(\mathfrak{n}) $, $x_{1}\wedge \ldots \wedge  x_{p} \in \Lambda^{p}(\mathfrak{n})$ and $x \in \mathfrak{n}$. Then, \begin{eqnarray*}
\underline{\alpha}(xy_1y_2\ldots y_m \otimes x_{1}\wedge \ldots \wedge  x_{p}) & = & \sum_{j=1}^m x y_1\ldots \alpha(y_j)\ldots y_m \otimes x_{1}\wedge \ldots \wedge  x_{p} \\
& & + \alpha(x)y_1y_2\ldots y_m \otimes x_{1}\wedge \ldots \wedge  x_{p} \\
& & +  xy_1y_2\ldots y_m \underline{\alpha}(1\otimes x_{1}\wedge \ldots \wedge  x_{p})\\
& = & \alpha(x)y_1\ldots y_m\otimes x_{1}\wedge \ldots \wedge  x_{p}+ \\
& & x\underline{\alpha}(y_1\ldots y_m\otimes x_{1}\wedge \ldots \wedge  x_{p}).
\end{eqnarray*} This shows that condition (\ref{eq: comp cond}) is satisfied.
Now, using the definition of $\underline{\alpha}$, we see that this is the same expression as before.
We conclude that the maps $\underline{\alpha}$ indeed define a compatible action of $\mathfrak{h}$ on the Chevalley-Eilenberg complex of $\mathfrak{n}$.
\end{proof}
Next, we give four simple but useful lemmas that allow us to construct new compatible actions from already existing ones. The proofs of the first three lemmas are straightforward.
\begin{lemma} \label{lemma: homomorphism} Let $\mathfrak{h}_1\to \mathrm{Der}(\mathfrak{n})$ be a Lie algebra homomorphism and  suppose $\mathfrak{h}_1$ acts compatibly on a free $U(\mathfrak{n})$-resolution $\varepsilon : F \rightarrow k$.  If $\phi: \mathfrak{h}_2 \rightarrow \mathfrak{h}_1$ is a Lie algebra homomorphism, then $\underline{\alpha}(f)= \underline{(\phi(\alpha))}(f)$ defines a compatible action of $\mathfrak{h}_2$ on $F$.
\end{lemma}
\begin{lemma} \label{lemma: kernel product} Suppose that, for $i=1,2$, we have a Lie algebra homomorphism $\varphi_i : \mathfrak{h} \rightarrow \mathrm{Der}(\mathfrak{n}_i)$ such that $\mathfrak{h}$ acts compatibly on a free $U(\mathfrak{n}_i)$-resolution $\varepsilon_i : F_i \rightarrow k$. Then, considering the homomorphism $\varphi: \mathfrak{h} \rightarrow \mathrm{Der}(\mathfrak{n}_1)\oplus \mathrm{Der}(\mathfrak{n}_2) \hookrightarrow \mathrm{Der}(\mathfrak{n}_1 \oplus \mathfrak{n}_2)$, we obtain a compatible action of $\mathfrak{h}$ on the free $U(\mathfrak{n}_1 \oplus \mathfrak{n}_2)$-resolution $\varepsilon_1 \otimes \varepsilon_2 : F_1 \otimes_k F_2 \rightarrow k$ given by $\underline{\alpha}(f_1 \otimes f_2):= \underline{\alpha}(f_1)\otimes f_2 + f_2 \otimes \underline{\alpha}(f_2)$.
\end{lemma}
\begin{lemma} \label{lemma: quotient semi-direct product}
Let $\mathfrak{h}=\mathfrak{h}_1 \rtimes_{\rho} \mathfrak{h}_2$ and let $\varphi : \mathfrak{h} \rightarrow \mathrm{Der}(\mathfrak{n})$ be a Lie algebra homomorphism. For each $i=1, 2$, suppose $\mathfrak{h}_i$ acts compatibly through $\varphi$ on the free $U(\mathfrak{n})$-resolution $\varepsilon : F \rightarrow k$. If  $\underline{\rho(\alpha_2)(\alpha_1)}= \underline{\alpha_2}\circ \underline{\alpha_1}-\underline{\alpha_1}\circ \underline{\alpha_2}$ for all $(\alpha_1,\alpha_2) \in \mathfrak{h}$, then $\underline{(\alpha_1,\alpha_2)}(f)=\underline{\alpha_1}(f)+\underline{\alpha_2}(f)$ defines a compatible action of $\mathfrak{h}$ on $F$.
\end{lemma}
\begin{lemma} \label{ex: free product}
Suppose $\mathfrak{n}_1$ and $\mathfrak{n}_2$ are two Lie algebras and consider a Lie algebra homomorphism $\varphi : \mathfrak{h} \rightarrow \mathrm{Der}(\mathfrak{n}_1*\mathfrak{n}_2) $
such that for every $\alpha \in \mathfrak{h}$, we have $\varphi(\alpha)(\mathfrak{n}_i)\subset \mathfrak{n}_i$ for $i=1,2$.
Then
\begin{equation*}
\ldots \rightarrow U(\mathfrak{n}_1*\mathfrak{n}_2) \otimes (\Lambda^{p}(\mathfrak{n}_1)\oplus \Lambda^{p}(\mathfrak{n}_2)) \rightarrow \ldots  \rightarrow   U(\mathfrak{n}_1*\mathfrak{n}_2) \rightarrow k \rightarrow 0 .\end{equation*}
is a free $U(\mathfrak{n}_1*\mathfrak{n}_2)$-resolution of $k$ that allows a compatible action of $\mathfrak{h}$.
\end{lemma}
\begin{proof}
 Given two Lie algebras $\mathfrak{n}_1$ and $\mathfrak{n}_2$, we can consider their free product $\mathfrak{n}_1 * \mathfrak{n}_2$. If $J_i$ is the augmentation ideal of $\mathfrak{n}_i$ and $J$ is the augmentation ideal of $\mathfrak{n}_1 * \mathfrak{n}_2$, then $$J \cong  \Big(U(\mathfrak{n}_1*\mathfrak{n}_2)\otimes_{U(\mathfrak{n}_1)} J_1 \Big) \oplus \Big(U(\mathfrak{n}_1*\mathfrak{n}_2)\otimes_{U(\mathfrak{n}_2)} J_2 \Big)$$ as left $U(\mathfrak{n}_1*\mathfrak{n}_2)$-modules. Denote by $V_{\ast}(\mathfrak{n}_i)$ the Chevalley-Eilenberg resolution of $\mathfrak{n}_i$, for $i=1,2$. Then
\begin{equation} \label{eq: free product1}
\ldots \rightarrow V_{p}(\mathfrak{n}_i) \rightarrow V_{p-1}(\mathfrak{n}_i) \rightarrow \ldots \ldots V_1(\mathfrak{n}_i) \rightarrow J_i \rightarrow 0 \end{equation}
is a free $U(\mathfrak{n}_i)$-resolution of $J_i$, for $i=1,2$. Since $U(\mathfrak{n}_1 * \mathfrak{n}_2)$ is a free $U(\mathfrak{n}_i)$-module, applying $U(\mathfrak{n}_1 * \mathfrak{n}_2) \otimes_{U(\mathfrak{n}_i)}-$ to (\ref{eq: free product1}) yields an exact $U(\mathfrak{n}_1 * \mathfrak{n}_2)$-complex
\begin{equation} \label{eq: free product2}
\ldots \rightarrow U(\mathfrak{n}_1*\mathfrak{n}_2) \otimes \Lambda^{p}(\mathfrak{n}_i) \rightarrow U(\mathfrak{n}_1*\mathfrak{n}_2)\otimes \Lambda^{p-1}(\mathfrak{n}_i) \rightarrow \ldots  \rightarrow U(\mathfrak{n}_1 * \mathfrak{n}_2) \otimes_{U(\mathfrak{n}_i)} J_i \rightarrow 0 \end{equation}
for $i=1,2$. If we now take the direct sum of (\ref{eq: free product2}) for $i=1,2$, we obtain a free $U(\mathfrak{n}_1 * \mathfrak{n}_2)$-resolution of $J$ that we can extend to a free $U(\mathfrak{n}_1 * \mathfrak{n}_2)$-resolution of $k$, i.e.
\begin{equation} \label{eq: free product3}
\ldots \rightarrow U(\mathfrak{n}_1*\mathfrak{n}_2) \otimes (\Lambda^{p}(\mathfrak{n}_1)\oplus \Lambda^{p}(\mathfrak{n}_2)) \rightarrow \ldots  \rightarrow   U(\mathfrak{n}_1*\mathfrak{n}_2) \rightarrow k \rightarrow 0 .\end{equation}
Now suppose we have a Lie algebra homomorphism $\varphi: \mathfrak{h} \rightarrow \mathrm{Der}(\mathfrak{n}_1 * \mathfrak{n}_2)$ such that for every $\alpha \in \mathfrak{h}$, we have $\varphi(\alpha)(\mathfrak{n}_i)\subset \mathfrak{n}_i$ for $i=1,2$. Then $\mathfrak{h}$ acts compatibly on (\ref{eq: free product3}). Indeed, the compatible action on the $n$-th module of (\ref{eq: free product3}) can be defined on each of the two direct summands $U(\mathfrak{n}_1*\mathfrak{n}_2) \otimes \Lambda^{p}(\mathfrak{n}_i)$  as in Proposition \ref{prop: comp act on chev}.
\end{proof}
\begin{remark} \rm
Note that this lemma can easily be generalized to free products of more that two factors. Also, if one of the factors, say $\mathfrak{n}_2$, is free then we can replace the Chevalley-Eilenberg complex of $\mathfrak{n}_2$ by a resolution of the form (\ref{eq: resolution for free}).
\end{remark}
In the next examples, we will apply these lemmas to construct several useful compatible actions.
\begin{example}\rm \label{ex: free lie algebra}
Let $\mathfrak{f}_m$ be the free Lie algebra on $m$ generators $\{x_1, \ldots, x_m \}=X$. Then $k\{X\}$, the free $k$-algebra on $X$, is the universal enveloping algebra of $\mathfrak{f}_m$. Since the augmentation ideal $\mathrm{J}$ of $k\{X\}$ can be seen as an $m$-dimensional free $k\{X\}$-module,
\begin{equation} \label{eq: resolution for free} 0 \rightarrow \mathrm{J} \rightarrow k\{X\} \rightarrow k \rightarrow 0 \end{equation}
is a free $U(\mathfrak{f_m})$-resolution of $k$. Now, consider the universal split extension $\mathfrak{f}_m \rtimes \mathrm{Der}(\mathfrak{f_m})$ and take $\alpha \in \mathrm{Der}(\mathfrak{f_m})$. Then one can easily check that the $k$-linear map
\[ \underline{\alpha}: k\{X\} \rightarrow k\{X\} : \left\{\begin{array}{ccc}
                                                    r \in k &  \mapsto  & 0 \\
                                                    x_{i_1}x_{i_2}\ldots x_{i_p} & \mapsto & \sum_{j=1}^p  x_{i_1}\ldots \alpha(x_{i_j})\ldots x_{i_p}
                                                  \end{array}\right. \]
induces a compatible action of $\mathrm{Der}(\mathfrak{f_m})$ on (\ref{eq: resolution for free}). Hence, by Lemma \ref{lemma: homomorphism}, every split extension $\mathfrak{f}_n \rtimes \mathfrak{h}$ allows a compatible action of $\mathfrak{h}$ on (\ref{eq: resolution for free}).
\end{example}
\begin{example}\rm \label{ex: product of free lie algebras}
Let $\mathfrak{f}_1$ be the free Lie algebra on $m$ generators $\{x_1, \ldots, x_m \}=X$ and take $\mathfrak{f}_2$ to be the free Lie algebra on $n$ generators $\{y_1, \ldots, y_n \}=Y$. Now consider $(\mathfrak{f}_1 \oplus \mathfrak{f}_2)  \rtimes \mathrm{Der}(\mathfrak{f}_1 \oplus \mathfrak{f}_2)$. Since $\mathrm{Der}(\mathfrak{f}_m \oplus \mathfrak{f}_n) = \mathrm{Der}(\mathfrak{f}_m)\oplus \mathrm{Der}(\mathfrak{f}_n)$, we can use the previous example and Lemma $\ref{lemma: kernel product}$ to obtain a compatible action of $\mathrm{Der}(\mathfrak{f}_m \oplus \mathfrak{f}_n)$ on a free $U(\mathfrak{f}_1 \oplus \mathfrak{f}_2)$-resolution of $k$. Using Lemma \ref{lemma: homomorphism}, we see that any split extension with kernel $\mathfrak{f}_1 \oplus \mathfrak{f}_2$ admits a compatible action of this form.
\end{example}
\begin{example} \rm \label{ex: product free and abelian}
Let $\mathfrak{f}$ be the free Lie algebra on $m$ generators $\{x_1, \ldots, x_m \}=X$ and let $k^n$ be the $n$-dimensional abelian Lie algebra with $k$-basis $\{t_1,\ldots,t_n\}$. Recall that the universal enveloping algebra of $k^n$ equals the polynomial ring in $n$ variables $k[t_1,\ldots,t_n]$ and that we have a split short exact sequence
\[ 0 \rightarrow \mathrm{Der}(\mathfrak{f},k^n) \rightarrow \mathrm{Der}(\mathfrak{f} \oplus k^n) \rightarrow \mathrm{Der}(\mathfrak{f})\oplus \mathrm{Der}(k^n) \rightarrow 0. \]
Here, $\mathrm{Der}(\mathfrak{f},k^n)$ denotes the abelian Lie algebra of all $k$-linear maps from $\mathfrak{f}$ to $k^n$ that map $[\mathfrak{f},\mathfrak{f}]$ to zero.

Now, denote by $F_1$ the free $U(\mathfrak{f})$-resolution (\ref{eq: resolution for free}) and by $F_2$ the Chevalley-Eilenberg resolution of $k^n$. By Example \ref{ex: free lie algebra}, we have a compatible action of $\mathrm{Der(\mathfrak{f})}$ on $F_1$ and by Proposition \ref{prop: comp act on chev} we have a compatible action of $\mathrm{Der}(k^n)$ on $F_2$. Hence, it follows from Lemma \ref{lemma: kernel product} that we have a compatible action of $\mathrm{Der}(\mathfrak{f})\oplus \mathrm{Der}(k^n)$ on $F=F_1 \otimes F_2$. Next we will construct a compatible action of $\mathrm{Der}(\mathfrak{f},k^n)$ on $F$. To simplify our notation, we will first rewrite the complex $F$. Define \begin{eqnarray*}
& A_0 &=  k,\\
& A_p  &=    \Lambda^p(k^n)\oplus (\Lambda^{p-1}(k^n)\otimes\langle x_1,\ldots,x_m\rangle) \ \ \ \  \mbox{for} \ \ 1 \leq p \leq n,\\
& A_{n+1} &=  \Lambda^n(k^n)\otimes\langle x_1,\ldots , x_m\rangle,
\end{eqnarray*} where $\langle x_1,\ldots,x_m\rangle$ is the $m$-dimensional vector space with basis $X$. Then one can check that $F_p = U(\mathfrak{f} \oplus k^n)\otimes A_p$ for all $p \in \{0,\ldots,n+1\}$, with differentials given by \begin{eqnarray*}
d_1 : F_1 &\rightarrow& F_{0}: w \otimes (t_i,  x_k)  \mapsto  w(x_k,t_i)  \\
d_p : F_p &\rightarrow& F_{p-1}: w \otimes (t_{i_1}\wedge  \ldots \wedge t_{i_p},t_{j_1}\wedge  \ldots \wedge t_{j_{p-1}}\otimes x_k)  \mapsto \\ && \sum_{r=1}^p (-1)^{r+1}w(0,t_{i_r})\otimes(t_{i_1}\wedge  \ldots \wedge \widehat{t_{i_r}} \wedge \ldots \wedge t_{i_p},0)  \\ & &  + \sum_{s=1}^{p-1} (-1)^{s}w(0,t_{j_s})\otimes(0,t_{j_1}\wedge  \ldots \wedge \widehat{t_{j_s}} \wedge \ldots \wedge t_{j_{p-1}}\otimes x_k)  \\  & &  + w(x_k,0)\otimes (t_{j_1}\wedge  \ldots \wedge t_{j_{p-1}},0) \end{eqnarray*}
for all $p \in \{2,\ldots,n\}$ and \begin{eqnarray*}d_{n+1} : F_{n+1} &\rightarrow&  F_{n}: w \otimes (t_{1}\wedge  \ldots \wedge t_{n}\otimes x_k)  \mapsto  \\ &&
 \sum_{j=1}^n (-1)^{j}w(0,t_{j})\otimes(0,t_{1}\wedge  \ldots \wedge \widehat{t_{j}} \wedge \ldots \wedge t_{n}\otimes x_k)  \\  & &  + w(x_k,0)\otimes (t_{1}\wedge  \ldots \wedge t_{n},0).\end{eqnarray*}
 
Now, take $\alpha \in \mathrm{Der}(\mathfrak{f},k^n)\subseteq \mathrm{Der}(\mathfrak{f}\oplus k^n)$ and $a_i \in U(\mathfrak{f} \oplus k^n)$, then one can check that the maps
\begin{eqnarray*}
\underline{\alpha}: F_{0} &\rightarrow & F_{0} : a_1a_2 \ldots a_r   \mapsto  \sum_{s=1}^r a_1\ldots \alpha(a_s)\ldots a_r \\
\underline{\alpha}: F_p  &\rightarrow &   F_p : a_1a_2 \ldots a_r \otimes (t_{i_1}\wedge  \ldots \wedge t_{i_p},t_{j_1}\wedge  \ldots \wedge t_{j_{p-1}}\otimes x_k)  \mapsto \\
&&\sum_{s=1}^r a_1\ldots \alpha(a_s)\ldots a_r \otimes (t_{i_1}\wedge  \ldots \wedge t_{i_p},t_{j_1}\wedge  \ldots \wedge t_{j_{p-1}}\otimes x_k) \\
&& + (-1)^{p-1}a_1\ldots a_r \otimes(t_{j_1}\wedge  \ldots \wedge t_{j_{p-1}}\wedge \alpha(x_k),0) \\
\underline{\alpha}: F_{n+1} &\rightarrow & F_{n+1} : a_1a_2 \ldots a_r \otimes (t_{1}\wedge  \ldots \wedge t_{n}\otimes x_k)  \mapsto  \\
&& \sum_{s=1}^r a_1\ldots \alpha(a_s)\ldots a_r \otimes (t_{1}\wedge  \ldots \wedge t_{n}\otimes x_k)
\end{eqnarray*} define a compatible action of $\mathrm{Der}(\mathfrak{f},k^n)$ on $F$. Furthermore, one can verify that
$$\underline{\varphi_2\circ \alpha} - \underline{\alpha \circ \varphi_1} =  \underline{(\varphi_1, \varphi_2)}\circ \underline{\alpha} - \underline{\alpha} \circ \underline{(\varphi_1, \varphi_2)}$$ for all $(\varphi_1,\varphi_2) \in \mathrm{Der}(\mathfrak{f})\oplus \mathrm{Der}(k^n)$ and for all $\alpha \in \mathrm{Der}(\mathfrak{f},k^n)$
which means that the action of $\mathrm{Der}(\mathfrak{f},k^n)$ is compatible with the action of $\mathrm{Der}(\mathfrak{f})\oplus \mathrm{Der}(k^n)$ in the sense of Lemma \ref{lemma: quotient semi-direct product}. It now follows from Lemma \ref{lemma: quotient semi-direct product} that we have a compatible action of $\mathrm{Der}(\mathfrak{f}\oplus k^n)$ on $F$. Finally, Lemma \ref{lemma: homomorphism} implies that every split extension with kernel $\mathfrak{f} \oplus k^n$ admits a compatible action of this form.
\end{example}

\section{The Hochschild-Serre spectral sequence of a split extension}
Recall that a short exact sequence of Lie algebras
\begin{equation} 0 \rightarrow \mathfrak{n} \rightarrow \mathfrak{g} \xrightarrow{\pi} \mathfrak{h} \rightarrow 0 \label{eq: extension} \end{equation}
and a $\mathfrak{g}$-module $M$ give rise to a Hochschild-Serre spectral sequence. For a general treatment of spectral sequences we refer the reader to \cite{McClearly} and \cite{Weibel}. The Hochschild-Serre spectral sequence for Lie algebra extensions is discussed in \cite{Barnes2} and \cite{HochSerre}.

When the extension (\ref{eq: extension}) splits, we propose a modification to the construction of the Hochschild-Serre spectral sequence.
\begin{proposition} \label{prop: uberprop} Let
$0 \rightarrow \mathfrak{n} \rightarrow \mathfrak{g} \rightarrow \mathfrak{h} \rightarrow 0$
be a split extension of Lie algebras  and let $M$ be a $\mathfrak{g}$-module. If $\varepsilon_P:P \rightarrow k$ is a free $U(\mathfrak{h})$-resolution and $\varepsilon_F:F \rightarrow k$ is a free $U(\mathfrak{n})$-resolution that allows a compatible action of $\mathfrak{h}$, then this action defines a $\mathfrak{g}$-module structure on $F$ such that,
\[ \mathrm{H}^n(\mathfrak{g},M)=\mathrm{H}^n\Big(\mathrm{Hom}_{\mathfrak{h}}(P,\mathrm{Hom}_{\mathfrak{n}}(F,M))\Big) \]
for each $n$.
\end{proposition}
\begin{proof}
According to Lemma \ref{prop: freeres}, $\varepsilon_P \otimes \varepsilon_F: P\otimes_k F \rightarrow k$ is a free $U(\mathfrak{g})$-resolution. Therefore, $$\mathrm{H}^{\ast}(\mathfrak{g},M)= H^{\ast}(\mathrm{Hom}_{\mathfrak{g}}(P\otimes_k F,M)).$$
Also, by Lemma \ref{lem: gmod lemma}, we have
\[\mathrm{Hom}_{\mathfrak{g}}(P\otimes_k F,M) \cong \mathrm{Hom}_{\mathfrak{g}}(P,\mathrm{Hom}_k(F,M)).\]
Furthermore, since $\mathfrak{n}$ acts trivially on $P_p$ for each $p$, it follows from Lemmas \ref{lem: gmod lemma} and \ref{lem: gmod lemma2} that
\[ \mathrm{Hom}_{\mathfrak{g}}(P_q,\mathrm{Hom}_k(F_q,M))=\mathrm{Hom}_{\mathfrak{h}}(P_p,\mathrm{Hom}_{\mathfrak{n}}(F_q,M))\]
for all $p$ and $q$. We conclude that $\mathrm{H}^{\ast}(\mathfrak{g},M)$ can be calculated by taking the cohomology of $\mathrm{Hom}_{\mathfrak{h}}(P,\mathrm{Hom}_{\mathfrak{n}}(F,M))$.
\end{proof}
\indent Filtering by columns, we can obtain a canonically bounded filtration of the (total) Hom cochain complex $\mathrm{Hom}_{\mathfrak{h}}(P,\mathrm{Hom}_{\mathfrak{n}}(F,M))$. By constructing the spectral sequence associated to this filtration and using the proposition above, we obtain a convergent first quadrant spectral sequence
\begin{equation} E^{p,q}_2= \mathrm{H}^p(\mathfrak{h},\mathrm{H}^q(\mathfrak{n},M)) \Rightarrow \mathrm{H}^{p+q}(\mathfrak{g},M). \label{eq: spectral sequence} \end{equation}
Moreover, it is not difficult to see that this spectral sequence coincides with the Hochschild-Serre spectral sequence from the second page onward.
We will use this different construction of the Hochschild-Serre spectral sequence to prove a generalization of Theorem $2$ from \cite{Barnes}, but first we need a lemma.
\begin{lemma} Suppose $(C,d_h,d_v)$ is a first quadrant double complex with the vertical differential
$d_v^{p+1,q-1}: C^{p+1,q-1} \rightarrow C^{p+1,q}$
zero for some $p$ and $q$. Then the differentials $d^{p,q}_r$ and $d^{p-r+2,q+r-2}_r$, from the convergent first quadrant spectral sequence
\[ {}^{I}E_2^{p,q}=\mathrm{H}^p_{h}\mathrm{H}^q_v(C) \Rightarrow \mathrm{H}^{p+q}(\mathscr{C}), \]
obtained by filtering $C$ columnwise, are zero for all $r\geq 2$.
\end{lemma}
\begin{proof}
Recall that $\mathscr{C}$ is the cochain complex with
$ \mathscr{C}^n = \bigoplus_{k+l=n}C^{k,l}$,
and the differential $d$ is defined by $d_h+d_v$. The filtration of $\mathscr{C}$ is given by
$ F^p\mathscr{C}^n = \bigoplus_{\substack{k+l=n\\ k \geq p}}C^{k,l}$.
By definition we have
$E_r^{p,q}= {Z^{p,q}_r}/(Z^{p+1,q-1}_{r-1}+B^{p,q}_{r-1})$,
with
\begin{eqnarray*} Z^{p,q}_r  & = & F^p\mathscr{C}^{p+q}\cap d^{-1}\Big(F^{p+r}\mathscr{C}^{p+q+1}\Big),\\
 B^{p,q}_r  & = & F^p\mathscr{C}^{p+q}\cap d\Big(F^{p-r}\mathscr{C}^{p+q-1}\Big). \end{eqnarray*}
 Also, the differentials $d^{p,q}_r : E_r^{p,q} \rightarrow E_r^{p+r,q-r+1}$ are induced by the restriction of $d$ to $Z^{p,q}_r$.

 Now, let $[x] \in E_r^{p,q}$ where $x \in Z_r^{p,q}$. We can write $x=f+x'$ with $f \in C^{p,q}$ and $x' \in F^{p+1}\mathscr{C}^{p+q}$. Since $d_v^{p+1,q-1}=0$, we have $d(x)=d(x')$ (if $r\geq 2$). This means that $d(x) \in F^{p+r}\mathscr{C}^{p+q+1}\cap d\Big(F^{p+1}\mathscr{C}^{p+q}\Big)= B^{p+r,q-r+1}_{r-1}$ showing that $d^{p,q}_r([x])=0$. Since $[x]$ and $r$ are arbitrary, we conclude that $d^{p,q}_r=0$ for all $r\geq 0$.

Similarly, take $[x] \in E_r^{p-r+2,q+r-2}$ where $x \in Z_r^{p-r+2,q+r-2} \subset F^{p-r+2}\mathscr{C}^{p+q}$. Then $d^{p-r+2,q+r-2}_r([x])=[d(x)] \in E_r^{p+2,q-1}$. We will show that $d(x) \in B_{r-1}^{p+2,q-1}$. Denote by $x'$ the image of $x$ under the projection of $ F^{p-r+2}\mathscr{C}^{p+q}$ onto $ F^{p+1}\mathscr{C}^{p+q}$. Because $d_v^{p+1,q-1}=0$, one can easily verify that $d(x)=d(x')$. But this implies that $d(x) \in  B_{r-1}^{p+2,q-1}$, because $F^{p+1}\mathscr{C}^{p+q}\subset F^{p-r+3}\mathscr{C}^{p+q}$ for $r \geq 2$. By definition of $E_r^{p+2,q-1}$, this means that $d^{p-r+2,q+r-2}_r([x])=0$. Since $[x]$ and $r$ are arbitrary, we conclude that $d^{p-r+2,q+r-2}_r=0$ for all $r\geq 0$.
\end{proof}

\begin{theorem} \label{th: collapse}
Suppose
$0 \rightarrow \mathfrak{n} \rightarrow \mathfrak{g} \rightarrow \mathfrak{h} \rightarrow 0$
is a split extension of Lie algebras. Let M be a $\mathfrak{g}$-module and denote by $(E_r,d_r)$ the associated Hochschild-Serre spectral sequence.
If $\mathfrak{h}$ acts compatibly on a free  $U(\mathfrak{n})$-resolution $F$ such that the differential
\[ d^{q-1}: \mathrm{Hom}_{\mathfrak{n}}(F_{q-1},M) \rightarrow \mathrm{Hom}_{\mathfrak{n}}(F_q,M) \]
is zero, then $d^{p,q}_r$ and $d^{p,q+r-2}_r$ are zero for all $p$ and all $r\geq 2$.
\end{theorem}
\begin{proof}
If
$d^{q-1}: \mathrm{Hom}_{\mathfrak{n}}(F_{q-1},M) \rightarrow \mathrm{Hom}_{\mathfrak{n}}(F_q,M)$
is zero, then the vertical differentials $d_v^{p,q-1}$ of the double complex
$\mathrm{Hom}_{\mathfrak{h}}(P,\mathrm{Hom}_{\mathfrak{n}}(F,M))$ are zero for all $p$. It now follows from the previous lemma that $d^{p,q}_r$ and $d^{p,q+r-2}_r$ are zero for all $p$ and all $r\geq 2$.
\end{proof}
\begin{theorem} \label{cor: collapse} Let
$0 \rightarrow \mathfrak{n} \rightarrow \mathfrak{g} \rightarrow \mathfrak{h} \rightarrow 0$ be the split extension
determined by $\varphi: \mathfrak{h} \rightarrow \mathrm{Der}(\mathfrak{n})$ and let $M$ be a $\mathfrak{g}$-module such that $M^{\mathfrak{n}}=M$. Then the Hochschild-Serre spectral sequence associated to this extension with coefficients in $M$ collapses at $E_2$ in the following cases
\begin{itemize}
\item[(a)] $\mathfrak{n}=\mathfrak{n}_1 \oplus \mathfrak{n}_2$, where $\mathfrak{n}_1$ is either abelian or free and $\mathfrak{n}_2$ is either abelian or free;

\smallskip

\item[(b)] $\mathfrak{n}=\mathfrak{n}_1 * \mathfrak{n}_2* \ldots * \mathfrak{n}_k$, where each of the $\mathfrak{n}_i$ is either abelian or free and $\varphi(\alpha)(\mathfrak{n}_i)\subseteq \mathfrak{n}_i$ $\forall \alpha \in \mathfrak{h}$ and $i=1,\ldots, k$.
\end{itemize}
\end{theorem}
\begin{proof}
First suppose that $\mathfrak{n}$ is abelian. We know that $\mathfrak{h}$ acts compatibly on the Chevalley-Eilenberg complex $V(\mathfrak{n})$ of $\mathfrak{n}$. Since $\mathfrak{n}$ acts trivially on $M$, the differential
$$d^{q-1}: \mathrm{Hom}_{\mathfrak{n}}(V_{q-1}(\mathfrak{n}),M)\rightarrow \mathrm{Hom}_{\mathfrak{n}}(V_q(\mathfrak{n}),M)$$
is zero for all $q$. Hence, it follows that $d_r^{p,q}=0$ for all $p,q$ and $r\geq2$ which means that the spectral sequence collapses at $E_2$.

Now, assume that $\mathfrak{n}=\mathfrak{n}_1 \oplus \mathfrak{n}_2$  with $\mathfrak{n}_1$ and $\mathfrak{n}_2$ both free. In Example \ref{ex: product of free lie algebras}, we constructed a resolution for $\mathfrak{n}$ that allows a compatible action in any case and one easily checks, by for example using the K\"{u}nneth formula, that this resolution has zero differentials after applying $\mathrm{Hom}_{\mathfrak{n}}(-,M)$, when $\mathfrak{n}$ acts trivially on $M$. So, just as before we obtain the desired collapse. \\
If $\mathfrak{n}=\mathfrak{n}_1 \oplus \mathfrak{n}_2$, with $\mathfrak{n}_1$ free and $\mathfrak{n}_2$ abelian, then Example \ref{ex: product free and abelian} provides a free resolution of $\mathfrak{n}$ with compatible action that has zero differentials after applying $\mathrm{Hom}_{\mathfrak{n}}(-,M)$ because $\mathfrak{n}$ acts trivially on $M$. Thus, the collapse follows and case $(a)$ of the corollary is proven.

Part (b) is proven similarly by considering the compatible actions constructed in Lemma \ref{ex: free product}.
\end{proof}

Suppose that (\ref{eq: extension}) is a split extension with a finite dimensional kernel and consider its associated Hochschild-Serre spectral sequence with
coefficients in a $\mathfrak{g}$-module $M$,
\[ E^{p,q}_2= \mathrm{H}^p(\mathfrak{h},\mathrm{H}^q(\mathfrak{n},M)) \Rightarrow \mathrm{H}^{p+q}(\mathfrak{g},M). \]
It is clear that at some page $t$ the Hochschild-Serre spectral sequence will collapse, i.e. $E_r=E_{\infty}$ for all $r\geq t$. We define the \emph{length} $l$ of the spectral sequence to be the smallest $t$ for which $E_t=E_{\infty}$. This means that $d_r=0$ for all $r\geq l$, but $d_{l-1}\neq 0$. Using the previous theorem we can now prove the following.
\begin{theorem}  Suppose $0 \rightarrow \mathfrak{n} \rightarrow \mathfrak{g} \rightarrow \mathfrak{h} \rightarrow 0$
is a split extension of Lie algebras such that $\dim_k(\mathfrak{\mathfrak{n}})=m < \infty$ . Denote by $(E_r,d_r)$ the associated Hochschild-Serre spectral sequence with coefficients in the $\mathfrak{g}$-module $M$.
If $\mathfrak{n}$ acts trivially on $M$, then
\begin{itemize}
\item[(a)] $d_r^{p, m}=0$ for all $p$ and all $r\geq 2$;

\smallskip

\item[(b)] $l \leq \max{\{2, m\}}$;

\smallskip

\item[(c)] $\mathrm{H}^p(\mathfrak{h}, \mathrm{H}^{m}(\mathfrak{n},M))\oplus\mathrm{H}^{p+m}(\mathfrak{h}, M)\subseteq \mathrm{H}^{p+m}(\mathfrak{g},M)$ for all $p$.
\end{itemize}
\end{theorem}
\begin{proof}
%
Since $\mathfrak{n}$ acts trivially on $M$, either $\mathrm{H}^m(\mathfrak{n},M)=0$ or $\mathrm{H}^m(\mathfrak{n},M)\cong M$. If $\mathrm{H}^m(\mathfrak{n},M)=0$, then $E^{p,m}_r=0$ for all $p$ and all $r\geq1$. This of course implies $d_r^{p,m}=0$ for all $p$ and all $r\geq 2$. If $\mathrm{H}^m(\mathfrak{n},M)=M$, then $ d^{m-1}: \mathrm{Hom}_{\mathfrak{n}}((V_{m-1}(\mathfrak{n}),M) \rightarrow \mathrm{Hom}_{\mathfrak{n}}((V_{m}(\mathfrak{n}),M)$ is zero. Since we always have a compatible action on the Chevalley-Eilenberg complex, Theorem \ref{th: collapse} implies that $d_r^{p,m}=0$ for all $p$ and all $r\geq 2$, so part (a) is proven.

Since $\mathfrak{n}$ acts trivially on $M$, we know that the differential $d^{0}:
\mathrm{Hom}_{\mathfrak{n}}((V_{0}(\mathfrak{n}),M) \rightarrow \mathrm{Hom}_{\mathfrak{n}}((V_{1}(\mathfrak{n}),M)$ is zero. It follows that all differentials $d_r$, for $r \geq 2$, that land on the bottom row of the spectral sequence are also zero. We conclude that $l \leq \max{\{2,m\}}$. This finishes (b).

A priori we have $E_{\infty}^{p,m}\oplus E_{\infty}^{p+m, 0}\subseteq \mathrm{H}^{p+m}(\mathfrak{g},M)$ and $E_{\infty}^{p+m, 0} =\mathrm{H}^{p+m}(\mathfrak{h},M)$ for all $p$. By part (a), $E_{\infty}^{p,m}=E_{m+1}^{p,m}= \dots = E_2^{p,m}$ for all $p$ and  $E_2^{p,m}\cong  \mathrm{H}^p(\mathfrak{h}, \mathrm{H}^{m}(\mathfrak{n},M))$. This proves part (c).
\end{proof}

\begin{remark} \rm
\begin{itemize}\item[]
\item[-] Since the extension splits and $\mathfrak{n}$ acts trivially on $M$, we know that the homomorphisms $\mathrm{H}^p(\mathfrak{h},M) \rightarrow  \mathrm{H}^p(\mathfrak{g},M)$ are injective for every $p$. This is another way to see that all differentials $d_r$, for $r \geq 2$, that land on the bottom row of the spectral sequence are zero.
\item[-] In \cite{Barnes}, Barnes shows that the spectral sequence of split extensions of finite dimensional Lie algebra with abelian kernel collapses at the second page if the kernel acts trivially in the coefficients. Case (a) of Theorem \ref{cor: collapse} is a generalization of this result.
\end{itemize}
\end{remark}

\end{document}